\documentclass[12pt]{amsart}
\usepackage{amscd,amssymb,amsthm,amsmath,amssymb,enumerate}
\usepackage[matrix,arrow]{xy}

\theoremstyle{definition}
\newtheorem{example}[equation]{Example}
\newtheorem*{example*}{Example}
\newtheorem{definition}[equation]{Definition}
\newtheorem{theorem}[equation]{Theorem}
\newtheorem{lemma}[equation]{Lemma}
\newtheorem{corollary}[equation]{Corollary}

\newtheorem*{conjecture*}{Conjecture}
\newtheorem{question}[equation]{Question}
\newtheorem*{question*}{Question}

\newtheorem*{problem*}{Problem}
\newtheorem*{theorem*}{Theorem}

\theoremstyle{remark}
\newtheorem{remark}[equation]{Remark}
\newtheorem*{remark*}{Remark}

\makeatletter\@addtoreset{equation}{section} \makeatother

\address{\emph{Ivan Cheltsov}
\newline
\textnormal{University of Edinburgh,  Scotland}
\newline
\textnormal{HSE University, Russian Federation}
\newline
\textnormal{\texttt{I.Cheltsov@ed.ac.uk}}}

\author{Ivan Cheltsov}

\title{Cylinders in rational surfaces}

\thanks{All varieties are assumed to be algebraic, projective and defined over $\mathbb{C}$.}

\begin{document}

\begin{abstract}
We answer a question of Ciro Ciliberto about cylinders in rational surfaces
which are obtained by blowing up the plane at points in general position.
\end{abstract}

\sloppy

\maketitle

\section{Introduction}
\label{section:intro}

Let $S$ be a  smooth rational surface.
A \emph{cylinder} in $S$ is an open subset $U\subset S$ such that
$U\cong\mathbb{C}^1\times Z$ for an affine curve $Z$. The surface $S$  contains many cylinders, and it seems hopeless to describe all of them.
Instead, one can consider a similar problem for \emph{polarized} surfaces (see \cite{KPZ11a,KPZ12a,KPZ12b,CheltsovParkWon1,CheltsovParkWon3,MarquandWon}).
To~describe it, fix an ample $\mathbb{Q}$-divisor $A$ on the surface $S$.

\begin{definition}
\label{definition:polar-cylinder}
An $A$-polar cylinder in $S$ is a Zariski open subset $U$ in $S$ such that
\begin{itemize}
\item[(C)] $U\cong\mathbb{C}^1\times Z$ for some affine curve $Z$, i.e., $U$ is a cylinder in $S$,
\item[(P)] there is an effective $\mathbb{Q}$-divisor $D$ on $S$ such that $D\sim_{\mathbb{Q}} A$ and $U=S\setminus\mathrm{Supp}(D)$.
\end{itemize}
\end{definition}

One can always choose an ample divisor $A$ such that $S$ contains an $A$-polar cylinder.
This follows from \cite[Proposition~3.13]{KPZ11a}.
One the other hand, we have

\begin{theorem}[{\cite{KPZ12b,CheltsovParkWon1,CheltsovParkWon3}}]
\label{theorem:del-Pezzo}
Let $S_d$ be a smooth del Pezzo surface of degree \mbox{$d=K_{S_d}^2$},
Then the following assertions hold:
\begin{enumerate}
\item The surface $S_d$ contains a $(-K_{S_d})$-polar cylinder if and only if $d\geqslant 4$.
\item If $d\geqslant 4$, then $S_d$ contains an $H$-polar cylinder for every ample $\mathbb{Q}$-divisor $H$ on $S_d$;
\item If $d=3$, then $S_d$ contains an $H$-polar cylinder for every ample $\mathbb{Q}$-divisor $H$ on $S_d$ such that $H\not\in\mathbb{Q}_{>0}[-K_{S_d}]$.
\end{enumerate}
\end{theorem}

The paper \cite{CheltsovParkWon3} also contains one relevant result for del Pezzo surfaces of degree $1$ and $2$.
To describe this result, let
$$
\mu_A=\mathrm{inf}\Big\{\lambda\in\mathbb{Q}_{>0}\ \Big|\ \text{the $\mathbb{Q}$-divisor}\ K_{S}+\lambda A\ \text{is pseudo-effective}\Big\}\in\mathbb{Q}.
$$
The number $\mu_A$ is known as the Fujita invariant, pseudo-effective threshold or the spectral value of the divisor $A$ (see \cite{LTT,Sakai}).
Let $\Delta_{A}$ be the smallest extremal face of the Mori cone $\overline{\mathbb{NE}(S)}$ that contains $K_{S}+\mu_A A$.
Denote by $r_A$ the dimension of the face $\Delta_{A}$.
Observe that $r_A=0$ if and only if $S$ is a smooth del Pezzo surface and $\mu_A A\sim_{\mathbb{Q}}-K_S$.
The number $r_A$ is known as the Fujita rank of the divisor~$A$ (see \cite{CheltsovParkWon3}).

\begin{theorem}[{\cite{CheltsovParkWon3}}]
\label{theorem:del-Pezzo-3}
Let $S_d$ be a smooth del Pezzo surface of degree~$d=K_{S_d}^2$,
let $H$ be an ample $\mathbb{Q}$-divisor on  $S_d$, and let $r_H$ be the  Fujita rank of the divisor $H$.
Suppose that $r_H+d\leqslant 3$.
Then $S_d$ does not contain $H$-polar cylinders.
\end{theorem}

During the conference \emph{Complex affine geometry, hyperbolicity and complex analysis},
which was held in Grenoble in October 2016, Ciro Ciliberto asked

\begin{question}
\label{question:Ciliberto}
Let $S$ be a rational surface that is obtained from $\mathbb{P}^2$ by blowing up points in general position,
and let $A$ be an ample $\mathbb{Q}$-divisor on $S$ such that  $r_A+K_S^2\leqslant 3$.
Is it true that $S$ does not contain $A$-polar cylinders?
\end{question}

Ciliberto also suggested to consider Question~\ref{question:Ciliberto} modulo \cite[Conjecture~2.3]{Ciliberto}.
In this paper, we show that the answer to Question~\ref{question:Ciliberto} is Yes. To be precise, we prove

\begin{theorem}
\label{theorem:main}
Let $S$ be a smooth rational surface that satisfies the following generality condition:
\begin{equation}
\label{equation:condition}\tag{$\bigstar$}
\text{the self-intersection of every smooth rational curve in $S$ is at least $-1$.}
\end{equation}
Let $A$ be an ample $\mathbb{Q}$-divisor on~$S$, and let $r_A$ be the Fujita rank of the divisor~$A$.
Suppose that $r_A+K_S^2\leqslant 3$.
Then $S$ does not contain $A$-polar cylinders.
\end{theorem}

By \cite[Proposition~2.4]{deFernex}, rational surfaces obtained by blowing up $\mathbb{P}^2$ at points in general position satisfy \eqref{equation:condition}.
Thus, the answer to Question~\ref{question:Ciliberto} is Yes.

\begin{remark}
\label{remark:generalization}
Smooth del Pezzo surfaces satisfy \eqref{equation:condition}.
Moreover, if $K_S^2\geqslant 1$, then the divisor $-K_S$ is ample if and only if $S$ satisfies \eqref{equation:condition}.
This shows that Theorem~\ref{theorem:main} is a generalization of Theorem~\ref{theorem:del-Pezzo-3}.
\end{remark}

By \cite[Corollary~3.2]{KPZ12a}, Theorem~\ref{theorem:main} implies

\begin{corollary}
\label{corollary:main}
Let $S$ be a smooth rational surface that satisfies \eqref{equation:condition},
let $A$ be an ample $\mathbb{Z}$-divisor on the surface $S$, let $r_A$ be the Fujita rank of the divisor~$A$, and let
$$
V=\mathrm{Spec}\Bigg(\bigoplus_{n\geqslant 0}H^0\Big(S,\mathcal{O}_S\big(nA\big)\Big)\Bigg).
$$
Suppose that $r_A+K_S^2\leqslant 3$. Then $V$ does not admit an effective action of the additive group $\mathbb{C}_{+}$.
\end{corollary}

The following example shows that the inequality $r_A+K_S^2\leqslant 3$ in Theorem~\ref{theorem:main} is sharp.

\begin{example}
\label{example:auxiliary}
Let $S$ be a rational surface that satisfies \eqref{equation:condition}.
Suppose that $K_S^2\leqslant 3$.
Then~there exists a blow down $f\colon S\to\mathbb{P}^2$ of $9-K_{S}^2$ different points.
Put~$k=4-K_{S}^2\geqslant 1$.
Let $E_1$, $E_2$, $E_3$, $E_4$, $E_5$, $G_1,\ldots,G_k$ be the exceptional curves of the blow up~$f$,
let~$\mathcal{C}$ be the unique conic in $\mathbb{P}^2$ that passes through $f(E_1)$, $f(E_2)$, $f(E_3)$, $f(E_4)$, $f(E_5)$,
let $L$ be a general line in $\mathbb{P}^2$ tangent to $\mathcal{C}$, and let $\mathcal{P}$ be the pencil generated by $\mathcal{C}$~and~$2L$.
Denote by $C_i$ the conic in $\mathcal{P}$ that contains $f(G_i)$.
Then $\mathbb{P}^{2}\setminus(\mathcal{C}\cup L\cup C_1\cup\cdots\cup C_k)$ is~a~cylinder.
Denote the proper transforms of $\mathcal{C}$ and $L$ on $S$ by $\widetilde{\mathcal{C}}$ and $\widetilde{L}$, respectively.
Similarly, denote by $\widetilde{C}_i$ the proper transform of the conic $C_i$ on the surface $S$.
Then
$$
S\setminus \Big(\widetilde{\mathcal{C}}\cup\widetilde{L}\cup E_1\cup\cdots\cup E_5\cup\widetilde{C}_1\cup\cdots\cup \widetilde{C}_k\cup G_1\cup\cdots\cup G_k\Big)\cong\mathbb{P}^{2}\setminus\Big(\mathcal{C}\cup L\cup C_1\cup\cdots\cup C_k\Big).
$$
Let $\epsilon_1$, $\epsilon_2$ and $x$ be rational numbers such that $\frac{1}{2}>\epsilon_1>\frac{\epsilon_2}{2}>0$ and $1>x>1-\frac{1-2\epsilon_1}{2k}$.
Let~$A=-K_{S}+x(G_1+\cdots+G_k)$. Then $A$ is ample and $r_A=k$, since
$$
A\sim_{\mathbb{Q}} \Big(1+\epsilon_1-\frac{\epsilon_2}{2}\Big)\widetilde{\mathcal{C}}+\epsilon_2\widetilde{L}+\Big(\epsilon_1-\frac{\epsilon_2}{2}\Big)\sum_{i=1}^{5}E_i+\frac{1-2\epsilon_1}{2k}\sum_{i=1}^{k}\widetilde{C}_i+\Big(x+\frac{1-2\epsilon_1}{2k}-1\Big)\sum_{i=1}^{k}G_i.
$$
Thus, the surface $S$ contains an $A$-polar cylinder, and $r_A+K_{S}^2=4$.
\end{example}

The following example shows that the inequality $r_A+K_S^2\geqslant 4$ does not always imply the existence of $A$-polar cylinders in $S$.

\begin{example}
\label{example:nine-points}
Let $f\colon S\to\mathbb{P}^2$ be a blow up of $9$ points such that $|-K_{S}|$ is a base point free pencil.
Suppose that all curves in the pencil $|-K_{S}|$ are irreducible.
Then $S$ satisfies \eqref{equation:condition}.
Suppose, in addition, that all singular curves in the pencil $|-K_{S}|$ do not have cusps.
Let~$E_1$, $E_2$, $E_3$ and $E_4$ be any four $f$-exceptional curves.
Fix $x\in\mathbb{Q}$ such that $0<x<1$.
Let $A=-K_S+x(E_1+E_2+E_3+E_4)$. Then $A$ is ample. Moreover, we have~$r_A=4$.
Furthermore, if $x>\frac{7}{8}$, then it follows from Example~\ref{example:auxiliary} that $S$ contains an $A$-polar cylinder.
On the other hand, the surface $S$ does not contain $A$-polar cylinders for $x\leqslant\frac{1}{4}$ by Lemmas~\ref{lemma:nine-points}, \ref{lemma:main-1} and \ref{lemma:main-2}.
\end{example}

The following examples shows that that we cannot omit \eqref{equation:condition} in Theorem~\ref{theorem:main}.

\begin{example}
\label{example:dP-2}
Let $L_1$ and $L_2$ be two distinct lines in $\mathbb{P}^2$. Then $\mathbb{P}^2\setminus\big(L_1\cup L_2\big)\cong\mathbb{C}^1\times\mathbb{C}^{\ast}$.
Let~$P_1$ be a point in $L_1\setminus L_2$. Let $P_{2}$, $P_3$, $P_4$, $P_5$, $P_6$ and $P_7$ be general points in $L_2\setminus L_1$.
Let~$f\colon\widehat{S}\to\mathbb{P}^2$ be the blow up of these seven points $P_{1}$, $P_{2}$, $P_3$, $P_4$, $P_5$, $P_6$ and $P_7$.
Denote~by $F_1$, $F_{2}$, $F_3$, $F_4$, $F_5$, $F_6$ and $F_7$ the $f$-exceptional curves such that $f(F_i)=P_i$.
Let $g\colon\widetilde{S}\to\widehat{S}$ be the blow up of the point in $F_1$ contained in the proper transform~of~$L_1$.
Denote by $G$ the $g$-exceptional~curve. Let $\widetilde{F}_1$ be the proper transform on $\widetilde{S}$ of the curve $F_1$.
Let $h\colon\overline{S}\to\widetilde{S}$ be the blow up of the point $\widetilde{F}_1\cap G$.
Denote by $H$ the $h$-exceptional~curve.
Let $e\colon\mathcal{S}\to\widetilde{S}$ be the blow up of a general point in $H$.
Denote by $\mathcal{E}$ the $e$-exceptional~curve.
Denote the proper transforms of the~curves $H$, $G$, $F_1$, $F_{2}$, $F_3$, $F_4$, $F_5$, $F_6$, $F_7$, $L_{1}$ and $L_{2}$~on the surface $\mathcal{S}$
by $\mathcal{H}$, $\mathcal{G}$, $\mathcal{F}_1$, $\mathcal{F}_2$, $\mathcal{F}_3$,
$\mathcal{F}_4$, $\mathcal{F}_5$, $\mathcal{F}_6$, $\mathcal{F}_7$, $\mathcal{L}_{1}$ and $\mathcal{L}_{2}$, respectively.
Fix a positive rational number $\epsilon$ such that $\epsilon<\frac{1}{3}$.
Then
$$
-K_{\mathcal{S}}\sim_{\mathbb{Q}} (2-\epsilon)\mathcal{L}_1+(1+\epsilon)\mathcal{L}_{2}+(1-\epsilon)\mathcal{F}_1+\epsilon\sum_{i=2}^{7}\mathcal{F}_i+(2-2\epsilon)\mathcal{G}+(2-3\epsilon)\mathcal{H}+(1-3\epsilon)\mathcal{E}.
$$
We also have $\mathcal{S}\setminus(\mathcal{L}_{1}\cup\mathcal{L}_{2}\cup\mathcal{F}_1\cup\cdots\cup\mathcal{F}_7\cup\mathcal{G}\cup\mathcal{H}\cup\mathcal{E})\cong\mathbb{P}^2\setminus(L_1\cup L_2)$.
Let $\pi\colon\mathcal{S}\to S$ be the contraction of the curves $\mathcal{L}_{1}$, $\mathcal{G}$ and $\mathcal{H}$.
Then $S$ is a smooth surface.
We have $K_{S}^2=2$, the divisor $-K_{S}$ is nef and big, and $\pi(\mathcal{F}_1)\cdot\pi(\mathcal{F}_1)=\pi(\mathcal{L}_2)\cdot\pi(\mathcal{L}_2)=-2$.
In~particular, the surface $S$ does not satisfy \eqref{equation:condition}.
Let $L_{12}$ be the line in $\mathbb{P}^2$ that contains $P_1$ and $P_2$, and let
$\mathcal{L}_{12}$ be its proper transform~on~$\mathcal{S}$.
Fix a positive rational number $x$ such that $\epsilon>x>3\epsilon-1$.
Then
\begin{multline*}
-K_{\mathcal{S}}+x\mathcal{L}_{12}\sim_{\mathbb{Q}} (2-\epsilon)\mathcal{L}_1+(1+\epsilon)\mathcal{L}_{2}+(1-\epsilon)\mathcal{F}_1+(\epsilon-x)\mathcal{F}_2+\\
+\epsilon\Big(\mathcal{F}_3+\mathcal{F}_4+\mathcal{F}_5+\mathcal{F}_6+\mathcal{F}_7\Big)+(2+x-2\epsilon)\mathcal{G}+(2+x-3\epsilon)\mathcal{H}+(1+x-3\epsilon)\mathcal{E}.
\end{multline*}
Let $A=-K_S+x\pi(\mathcal{L}_{12})$. Then the divisor $A$ is ample and $r_A=1$, so that $r_A+K_S^2=3$.
On the other hand, the surface $S$ contains an $A$-polar cylinder, since
$$
A\sim_{\mathbb{Q}} (1+\epsilon)\pi\big(\mathcal{L}_{2}\big)+(1-\epsilon)\pi\big(\mathcal{F}_1\big)+(\epsilon-x)\pi\big(\mathcal{F}_2\big)+\epsilon\sum_{i=3}^{7}\pi\big(\mathcal{F}_i\big)+(1+x-3\epsilon)\pi\big(\mathcal{E}\big),
$$
and $S\setminus(\pi(\mathcal{L}_{2})\cup\pi(\mathcal{F}_1)\cup\cdots\cup\pi(\mathcal{F}_7)\cup\pi(\mathcal{E}))\cong\mathbb{C}^1\times\mathbb{C}^{\ast}$.
\end{example}

Let us describe the structure of this paper.
In Section~\ref{section:preliminaries}, we present results that are used in the proof of Theorem~\ref{theorem:main}.
In Section~\ref{section:three-lemmas}, we prove three lemmas that constitute the main part of the proof of Theorem~\ref{section:three-lemmas}.
In Section~\ref{section:proof}, we finish the proof of Theorem~\ref{theorem:main}.

\smallskip
\textbf{Acknowledgments.}
The author is grateful to Ciro Ciliberto for asking Question~\ref{question:Ciliberto}.
This work was partially supported within the framework of the HSE University Basic Research Program and the Russian Academic Excellence Project "5-100".

\section{Preliminaries}
\label{section:preliminaries}

Let $S$ be a smooth rational~surface,
and let $C_1,\ldots,C_n$ be irreducible curves on the surface $S$.
Fix non-negative rational numbers $\lambda_1,\ldots,\lambda_n$.
Let $D=\lambda_1C_1+\cdots+\lambda_n C_n$.
For consistency, we will use these notation throughout the whole paper.
In this section, we present a few of  well-known (local and global) results about $S$ and $D$
that will be used in the proof of Theorem~\ref{theorem:main}.
We start with

\begin{lemma}[{\cite[Theorem~4.57(2)]{KollarMori98}}]
\label{lemma:mult-1}
Let $P$ be a point in $S$.
Suppose the singularities of the log pair $(S,D)$ are not log canonical at $P$.
Then $\mathrm{mult}_P(D)>1$.
\end{lemma}

The following lemma is a special case of a much more general result, known as the inversion of adjunction (see \cite[Theorem~5.50]{KollarMori98}).

\begin{lemma}[{\cite[Corollary~5.57]{KollarMori98}}]
\label{lemma:inversion-of-adjunction}
Let $P$ be a smooth point of the curve $C_1$.
Suppose that $\lambda_1\leqslant 1$, and the log pair $(S,D)$ is not log canonical at $P$.
Let $\Delta=\lambda_2C_2+\cdots+\lambda_n C_n$.
Then
$C_1\cdot\Delta\geqslant(C_1\cdot\Delta)_P>1$.
\end{lemma}

We will also use the following (local) result.

\begin{lemma}[{\cite[Theorem~13]{Cheltsov}}]
\label{lemma:Vanya}
Let $P$ be a point in $C_1\cap C_2$.
Suppose that $\lambda_1\leqslant 1$ and $\lambda_2\leqslant 1$,
the curves $C_1$ and $C_2$ are smooth at $P$ and intersect transversally at $P$,
and the log pair $(S,D)$ is not log canonical at $P$.
Let $\Delta=\lambda_3C_3+\cdots+\lambda_n C_n$.
If $\mathrm{mult}_{P}(\Delta)\leqslant 1$, then
$(C_1\cdot\Delta)_P>1-\lambda_2$ or $(C_2\cdot\Delta)_P>1-\lambda_1$.
\end{lemma}

The following result has been used in Example~\ref{example:nine-points}.

\begin{lemma}
\label{lemma:nine-points}
In the assumptions and notation of Example~\ref{example:nine-points}, suppose that $D\sim_{\mathbb{Q}} A$ and \mbox{$x\leqslant\frac{1}{4}$}.
Then the log pair $(S,D)$ is log canonical.
\end{lemma}

\begin{proof}
Suppose that $(S,D)$ is not log canonical at some point $P\in S$.
Let $\mathcal{C}$ be the curve in the pencil $|-K_{S}|$ that contains $P$.
By assumption, the curve $\mathcal{C}$ is irreducible.
Moreover, its arithmetic genus is $1$, so that either it is smooth or it has one simple node,
because we assume that curves in the pencil $|-K_{S}|$ do not have cusps.

If $\mathcal{C}$ is not contained in $\mathrm{Supp}(D)$, then
$1\geqslant 4x=C_1\cdot\Delta\geqslant\mathrm{mult}_{P}(D)>1$ by Lemma~\ref{lemma:mult-1}.
This shows that $\mathcal{C}$ is contained in the support of the divisor $D$.
Without loss of generality, we may assume that $\mathcal{C}=C_1$ and $\lambda_1>1$.
Let $\Delta=\lambda_2C_2+\cdots+\lambda_n C_n$.

We claim that $\lambda_1<1$. Indeed, we have $C_1+x\big(E_1+E_2+E_3+E_4\big)\sim_{\mathbb{Q}}\lambda_1C_1+\Delta$,
and the intersection form of the curves $E_1$, $E_2$, $E_3$, $E_4$ is negative definite.
Thus, if~$\lambda_1\geqslant 1$, then $\lambda_1=1$ and $\Delta=x(E_1+E_2+E_3+E_4)$,
which is impossible, because the singularities of the log pair $(S,C_1+x(E_1+E_2+E_3+E_4))$ are log canonical,
since $C_1$ is either smooth or has one simple node (by assumption).

If $C_1$ is smooth at $P$, then $1\geqslant 4x=C_1\cdot\Delta\geqslant(C_1\cdot\Delta)_{P}>1$ by Lemma~\ref{lemma:inversion-of-adjunction},
so that the curve $C_1$ has a simple node at the point $P$. This implies that $P\not\in E_1\cup E_2\cup E_3\cup E_4$,
because $\mathcal{C}\cdot E_i=-K_S\cdot E_i=1$ for every $i$.

We may assume that one of the curves $E_1$, $E_2$, $E_3$, $E_4$ is not contained in $\mathrm{Supp}(\Delta)$,
since otherwise we can swap $D$ with the divisor
$(1+\mu)D-\mu(C_1+x\big(E_1+E_2+E_3+E_4))$ for an appropriate positive rational number $\mu$.
Without loss of generality, we may assume that $E_4\not\subset\mathrm{Supp}(\Delta)$. 
Then $1-x=E_4\cdot(\lambda_1C_1+\Delta)=\lambda_1+E_4\cdot\Delta\geqslant\lambda_1$.

Let $m=\mathrm{mult}_{P}(\Delta)$. Then $4x=C_1\cdot\Delta\geqslant 2m$, so that $m\leqslant 2x$.

Let $f\colon\widetilde{S}\to S$ be the blow up of the point $P$. Denote by $F$ the $f$-exceptional curve.
Denote by $\widetilde{C}_1$ and $\widetilde{\Delta}$ the proper transforms on $\widetilde{S}$ of the divisors $C_1$ and $\Delta$, respectively.
Then $(\widetilde{S},\lambda_1\widetilde{C}_1+\widetilde{\Delta}+(2\lambda_1+m-1)F)$ is not log canonical at some point $Q\in F$, since
$$
K_{\widetilde{S}}+\lambda_1\widetilde{C}_1+\widetilde{\Delta}+\big(2\lambda_1+m-1)F\sim_{\mathbb{Q}} f^{*}\big(K_S+D\big).
$$
Moreover, we have $2\lambda_1+m-1\leqslant 1$, since we already proved that $\lambda_1\leqslant 1-x$ and $m\leqslant 2x$.

If $Q\not\in\widetilde{C}_1$, then $(\widetilde{S},\widetilde{\Delta}+F)$ is not log canonical at $Q$, so that
$\frac{1}{2}\geqslant 2x\geqslant m=F\cdot\widetilde{\Delta}>1$ by Lemma~\ref{lemma:inversion-of-adjunction}. 
This shows that $Q\in\widetilde{C}_1$.

The curve $\widetilde{C}_1$ is smooth and intersects $F$ transversally at $Q$.
We know that $m\leqslant 2x\leqslant 1$.
Thus, we can apply Lemma~\ref{lemma:Vanya} to the log pair $(\widetilde{S},\lambda_1\widetilde{C}_1+\widetilde{\Delta}+(2\lambda_1+m-1)F)$.
Then
$$
4x-2m=\widetilde{\Delta}\cdot\widetilde{C}_1>2\Big(1-\big(2\lambda_1+m-1\big)\Big)
$$
or
$m=\widetilde{\Delta}\cdot F>2(1-\lambda_1)$.
This leads to a contradiction, since $m\leqslant 2x$ and $\lambda_1\leqslant 1-x$.
\end{proof}

In the proof of Theorem~\ref{theorem:main}, we will use the following (global) result.

\begin{theorem}[{\cite[Theorem~1.12]{{CheltsovParkWon1}}}]
\label{theorem:del-Pezzo-tigers}
Suppose that $S$ is a smooth del Pezzo surface such that $K_S^2\leqslant 3$, and $D\sim_{\mathbb{Q}}-K_{S}$.
Let $P$ be a point in $S$. Suppose that $(S,D)$ is not log canonical at the point~$P$.
Then the linear system $|-K_{S}|$ contains a unique curve $T$ such that $(S,T)$ is not log canonical at $P$.
Moreover, the support of the divisor $D$ contains all the irreducible components of the curve $T$.
\end{theorem}

Let~$U=S\setminus(C_1\cup\cdots\cup C_n)$.
Suppose that
$U\cong\mathbb{C}^1\times Z$ for an affine curve $Z$.

\begin{lemma}
\label{lemma:main-1}
One has $n\geqslant 10-K_S^2$.
\end{lemma}

\begin{proof}
This follows from the proof of \cite[Lemma~4.11]{KPZ11a}.
\end{proof}

The embeddings $Z\hookrightarrow\mathbb{P}^1$ and $\mathbb{C}^1\hookrightarrow\mathbb{P}^1$ induce a commutative diagram
$$
\xymatrix{\mathbb{P}^1\times\mathbb{P}^1\ar[dd]^{\bar{p}_{2}}&\mathbb{C}^1\times\mathbb{P}^1\ar@{_{(}->}[l]\ar[dd]^{p_{2}}&~\mathbb{C}^1\times Z\cong U\ar@{_{(}->}[l]\ar[d]^{p_Z}\ar@{^{(}->}[r] &S\ar@{-->}^{\psi}[dd]&& \\
&&Z\ar@{^{(}->}[rd] \ar@{_{(}->}[dl]&&&\mathcal{S}\ar[llu]_{\pi}\ar[dll]^{\phi} \\
\mathbb{P}^1\ar@{=}[r]&\mathbb{P}^1\ar@{=}[rr] &&\mathbb{P}^1 & &}
$$
where $p_Z$, $p_{2}$ and $\bar{p}_2$ are projections to the second factors, $\psi$ is the map induced by $p_Z$,
the~map $\pi$~is a birational morphism resolving the indeterminacy of $\psi$, and $\phi$ is a morphism.
Let $\mathcal{E}_1,\ldots,\mathcal{E}_m$ be the $\pi$-exceptional curves (if $\pi$ is an isomorphism, we let $m=0$).
Let $C$ be the section of the projection $\bar{p}_2$, which is the complement of~\mbox{$\mathbb{C}^1\times\mathbb{P}^1$}~in~$\mathbb{P}^1\times\mathbb{P}^1$.
Denote by $\mathcal{C}_1,\ldots,\mathcal{C}_n$ the proper transforms on $\mathcal{S}$ of the curves $C_1,\ldots,C_n$, respectively.
Similarly, denote by $\mathcal{C}$ the proper transform of the curve $C$ on the surface~$\mathcal{S}$.

\begin{lemma}
\label{lemma:main-2}
Suppose that $K_S+D$ is pseudo-effective, and $\lambda_i<2$ for every $i$.
Then $\pi(\mathcal{C})$ is a point, and $(S,D)$ is not log canonical at $\pi(\mathcal{C})$.
\end{lemma}

\begin{proof}
By construction, a general fiber of the morphism $\phi$ is a smooth rational curve, and the curve $\mathcal{C}$ is its section.
Then $\mathcal{C}$ is either one of the curves $\mathcal{C}_1,\ldots,\mathcal{C}_n$ or one of the curves $\mathcal{E}_1,\ldots,\mathcal{E}_m$.
All the other curves among $\mathcal{C}_1,\ldots,\mathcal{C}_n$ and $\mathcal{E}_1,\ldots,\mathcal{E}_m$ are mapped by $\phi$ to points in $\mathbb{P}^1$.
Thus, without loss of generality, we may assume that $\mathcal{C}=\mathcal{C}_1$ or $\mathcal{C}=\mathcal{E}_m$.

There are rational numbers $\mu_1,\ldots,\mu_m$ such that
$$
K_{\mathcal{S}}+\sum_{i=1}^n\lambda_i\mathcal{C}_i+\sum_{i=1}^m\mu_i\mathcal{E}_i= \pi^*\big(K_{S}+D\big).
$$
Let $\mathcal{F}$ be a general fiber of the morphism $\phi$. If $\mathcal{C}=\mathcal{C}_1$, then
$$
-2+\lambda_1=\Big(K_{\mathcal{S}}+\sum_{i=1}^n\lambda_i\mathcal{C}_i+\sum_{i=1}^m\mu_i\mathcal{E}_i\Big)\cdot\mathcal{F}=\pi^*\big(K_{S}+D\big)\cdot\mathcal{F}=\big(K_{S}+D\big)\cdot\pi\big(\mathcal{F}\big)\geqslant 0,
$$
because $K_S+D$ is pseudo-effective. Thus, in this case $\lambda_1>2$, which is impossible by assumption.
Hence, we conclude that $\mathcal{C}=E_m$, so that $\pi(\mathcal{C})$ is a point.
Then
$$
-2+\mu_m=\Big(K_{\mathcal{S}}+\sum_{i=1}^n\lambda_i\mathcal{C}_i+\sum_{i=1}^m\mu_i\mathcal{E}_i\Big)\cdot\mathcal{F}=\pi^*\big(K_{S}+D\big)\cdot\mathcal{F}=\big(K_{S}+D\big)\cdot\pi\big(\mathcal{F}\big)\geqslant 0,
$$
because the divisor $K_S+D$ is pseudo-effective.
This shows that the singularities of the log pair $(S,D)$ are not log canonical at the point $\pi(\mathcal{C})$.
\end{proof}

\section{Three main lemmas}
\label{section:three-lemmas}

In this section, we prove three results that will be used later in the proof of Theorem~\ref{theorem:main} in Section~\ref{section:proof}.
These results are Lemmas~\ref{lemma:1}, \ref{lemma:2} and \ref{lemma:3} below.

Let $S$ be a smooth rational surface that satisfy \eqref{equation:condition},
let $C_1,\ldots,C_n$ be irreducible curves on the surface $S$,
let $U=S\setminus(C_1\cup\cdots\cup C_n)$, and let $D=\sum_{i=1}^n\lambda_iC_i$
for some non-negative rational numbers $\lambda_1,\ldots,\lambda_n$.
Suppose also that $S$ contains smooth rational curves $E_1,\ldots,E_{r}$ such that $E_i^2=-1$ for every~$i$,
and
$$
D\sim_{\mathbb{Q}}-K_S+\sum_{i=1}^{r}a_iE_i
$$
for some non-negative rational numbers $a_1,\ldots,a_r$.

\begin{remark}
\label{remark:ample-not-ample}
If  $D$ is ample, then $r$ is the Fujita rank of the divisor $D$.
However, in this section, we deliberately do not assume that $D$ is ample.
We hope that this will not make big confusion.
We must consider non-ample divisors here, because
Lemmas~\ref{lemma:1}, \ref{lemma:2} and \ref{lemma:3} can be applied also to non-ample divisors,
which is also used in their proofs.
\end{remark}

Let $g\colon S\to\overline{S}$ be a blow down of the curves $E_1,\ldots,E_{r}$,
let $\overline{C}_1=g(C_1),\ldots,\overline{C}_n=g(C_n)$, and let $\overline{D}=\lambda_1\overline{C}_1+\cdots+\lambda_n\overline{C}_n$.
Then $K_{\overline{S}}^2=r+K_S^2$ and $\overline{D}\sim_{\mathbb{Q}} -K_{\overline{S}}$.

\begin{remark}
\label{remark:del-Pezzo}
Since $S$ satisfies \eqref{equation:condition} by assumption, the surface $\overline{S}$ also satisfies \eqref{equation:condition}.
In particular, if $r+K_S^2\geqslant 1$, then $\overline{S}$ is a smooth del Pezzo surface by Remark~\ref{remark:generalization}.
\end{remark}

First, let us prove one auxiliary result:

\begin{lemma}
\label{lemma:easy}
Suppose that $C_i\ne E_j$ for all $i$ and $j$. Then the log pair $(S,D)$ is log canonical along $E_1\cup\cdots\cup E_{r}$.
\end{lemma}

\begin{proof}
Suppose that the log pair $(S,D)$ is not log canonical at some point $P\in E_1\cup\cdots\cup E_{r}$.
Then $\mathrm{mult}_{P}(D)>1$ by Lemma~\ref{lemma:mult-1}.
Thus, if $P\in E_1$, then $1\geqslant 1-a_1=D\cdot E_1>1$,
which is absurd. Similarly, we see that $P\not\in E_2\cup\cdots\cup E_{r}$.
\end{proof}

Recall that $U=S\setminus(C_1\cup\cdots\cup C_n)$, and $r$ is the number of $g$-exceptional curves.

\begin{lemma}
\label{lemma:1}
Suppose that $r+K_S^2=1$, and $\lambda_i>0$ for every $i$.
Then $U$ is not a cylinder.
\end{lemma}

\begin{proof}
We have $U=S\setminus\mathrm{Supp}(D)$, and $\overline{S}$ is a smooth del Pezzo surface by Remark~\ref{remark:del-Pezzo}.
If $K_S^2=1$, then $r=0$, so that $S\cong\overline{S}$ and $D\sim_{\mathbb{Q}} -K_{S}$.
In this case, if $U$ is a cylinder, then $U$ is a $(-K_S)$-polar cylinder,
which is impossible by Theorem~\ref{theorem:del-Pezzo}.
Therefore, we may assume that $K_S^2\leqslant 0$.
Let us prove the required assertion by induction on $K_S^2$.

Suppose first that $C_1=E_1$. Then there exists a commutative diagram
$$
\xymatrix{
&S\ar[dl]_{f}\ar[dr]^{g}& \\
\widehat{S}\ar[rr]^{h}&&\overline{S}}
$$
where $f\colon S\to\widehat{S}$ is a contraction of the curve $C_1=E_1$, and $h$ is a birational morphism.
Denote by $\widehat{E}_2,\ldots,\widehat{E}_{r}$ the proper transforms on $\widehat{S}$ of the curves $E_2,\ldots,E_{r}$, respectively.
Denote by $\widehat{C}_2,\ldots,\widehat{C}_n$ the proper transforms on $\widehat{S}$ of the curves $C_2,\ldots,C_n$, respectively.
Then $K_{\widehat{S}}^2=K_S^2+1$ and
$$
-K_{\widehat{S}}+\sum_{i=2}^{r}a_i\widehat{E}_i\sim_{\mathbb{Q}}\sum_{i=2}^n\lambda_i\widehat{C}_i.
$$
By induction, the subset $\widehat{S}\setminus\big(\widehat{C}_2\cup\cdots\cup\widehat{C}_n\big)\cong U$
is not a cylinder.
Thus, we may assume that~$C_1\ne E_1$.
Similarly, we may assume that $C_i\ne E_j$ for all possible $i$ and $j$,
which means that none of the curves $E_1,\ldots,E_{r}$ are contained in $\mathrm{Supp}(D)$.

Suppose that $U$ is a cylinder.
Then $n\geqslant 10-K_S^2\geqslant 10$ by Lemma~\ref{lemma:main-1},
and
$$
1=-K_{\overline{S}}\cdot\overline{D}=-K_{\overline{S}}\cdot\Big(\lambda_1\overline{C}_1+\cdots+\lambda_n\overline{C}_n\Big)\geqslant\sum_{i=1}^{n}\lambda_i,
$$
because the divisor $-K_{\overline{S}}$ is ample.
Thus, we see that $\lambda_i<1$ for~every~$i$.

By Lemma~\ref{lemma:main-2}, the surface $S$ contains a point $P$ such that the log pair $(S,D)$ is not log canonical at~$P$.
In the notations of Section~\ref{section:preliminaries}, the point $P$ is the point $\pi(\mathcal{C})$.
Let $\overline{P}=g(P)$. Then $(\overline{S},\overline{D})$ is not log canonical at $\overline{P}$,
because $P\not\in E_1\cup\cdots\cup E_{r}$ by Lemma~\ref{lemma:easy}.

By Theorem~\ref{theorem:del-Pezzo-tigers}, there is~a unique curve $\overline{T}\in |-K_{\overline{S}}|$ such that $(\overline{S},\overline{T})$ is not log canonical at the point~$\overline{P}$.
Note that $\overline{T}$ is irreducible.
Thus, Theorem~\ref{theorem:del-Pezzo-tigers} also implies that $\overline{T}$ is one of the curves $\overline{C}_1,\ldots,\overline{C}_n$.
Without loss of generality, we may assume that $\overline{T}=\overline{C}_1$.

The curve $\overline{T}=\overline{C}_1$ is singular at $\overline{P}$.
In fact, we can say more: this curve has a cuspidal singularity at $\overline{P}$, and it is smooth away from this point.
For every $i\in\{1,\ldots,r\}$, we let
$$
m_i=\left\{\aligned
&0\ \text{if}\ g(E_i)\not\in\overline{T},\\
&1\ \text{if}\ g(E_i)\in\overline{T}.\\
\endaligned
\right.
$$
Then
$$
C_1\sim g^{*}\big(\overline{C}_1\big)-\sum_{i=1}^{r}m_iE_i\sim -K_{S}+\sum_{i=1}^{r}(1-m_i)E_i.
$$

Let us replace $D$ by a divisor $(1+\mu)D-\mu C_1$ for an appropriate rational number $\mu>0$
such that the new divisor is effective and its support does not contain the curve $C_1$.
Let
$$
D^\prime=\frac{1}{1-\lambda_1}D-\frac{\lambda_1}{1-\lambda_1}C_1=\sum_{i=2}^n\frac{\lambda_i}{1-\lambda_1}C_i.
$$
Then $D^\prime$ is an effective divisor whose support does not contain the curve $C_1$.
On the other hand, we have
$$
D^\prime\sim_{\mathbb{Q}}-K_{S}+\sum_{i=1}^{r}\frac{a_i+(m_i-1)\lambda_1}{1-\lambda_1}E_i.
$$
Thus, if $\frac{a_i+(m_i-1)\lambda_1}{1-\lambda_1}\geqslant 0$ for every $i$,
then $(S,D^\prime)$ is not log canonical at $P$ by Lemma~\ref{lemma:main-2}.
In this case, the singularities of the log pair
$$
\Bigg(\overline{S},\sum_{i=2}^n\frac{\lambda_i}{1-\lambda_1}\overline{C}_i\Bigg)
$$
are not log canonical at the point $\overline{P}$, because $P\not\in E_1\cup\cdots\cup E_{r}$.
The latter is impossible by Theorem~\ref{theorem:del-Pezzo-tigers}.
Therefore, at least one rational number among
$$
\frac{a_1+(m_1-1)\lambda_1}{1-\lambda_1}, \frac{a_2+(m_2-1)\lambda_1}{1-\lambda_1}, \ldots, \frac{a_{r}+(m_{r}-1)\lambda_1}{1-\lambda_1}
$$
must be negative.

Without loss of generality, we may assume that there exists $k\leqslant r$ such that
$$
\frac{a_i+(m_i-1)\lambda_1}{1-\lambda_1}<0
$$
for every $i\leqslant k$, and $\frac{a_i+(m_i-1)\lambda_1}{1-\lambda_1}\geqslant 0$ for every $i>k$ (if any).
Then~$m_1=\cdots=m_k=0$. We may also assume that $a_1\leqslant\cdots\leqslant a_k$. Let
$$
D^{\prime\prime}=\frac{1}{1-a_1}D-\frac{a_1}{1-a_1}C_1=\frac{\lambda_1-a_1}{1-a_1}C_1+\sum_{i=2}^n\frac{\lambda_i}{1-a_1}C_i.
$$
Then $D^{\prime\prime}$ is an effective $\mathbb{Q}$-divisor such that
$$
D^{\prime\prime}\sim_{\mathbb{Q}}-K_{S}+\sum_{i=2}^{r}\frac{a_i-a_1(1-m_i)}{1-a_1}E_i=-K_S+\sum_{i=2}^{k}\frac{a_i-a_1}{1-a_1}E_i+\sum_{i=k+1}^{r}\frac{a_i-a_1(1-m_i)}{1-a_1}E_i.
$$
Note that $\frac{a_i-a_1(1-m_i)}{1-a_1}\geqslant 0$ for every possible $i>k$, because $a_1<\lambda_1$.

Let $e\colon\widetilde{S}\to\overline{S}$ be the blow up of the point $g(E_1)$, and let $\widetilde{E}_1$ be its exceptional curve.
Denote by $\widetilde{C}_1,\ldots,\widetilde{C}_n$ the proper transforms on $\widetilde{S}$ of the curves $C_1,\ldots,C_n$, respectively.
Likewise, denote by $\widetilde{D}^{\prime\prime}$ the proper transform of the divisor  $D^{\prime\prime}$ on the surface $\widetilde{S}$. Then
$$
\widetilde{D}^{\prime\prime}=\frac{\lambda_1-a_1}{1-a_1}\widetilde{C}_1+\sum_{i=2}^n\frac{\lambda_i}{1-a_1}\widetilde{C}_i\sim_{\mathbb{Q}} -K_{\widetilde{S}}.
$$

Since $g(E_1)\not\in\overline{T}$, the point $g(E_1)$ is not the base point of the pencil $|-K_{\overline{S}}|$.
Thus, the pencil $|-K_{\overline{S}}|$ contains a unique irreducible curve that passes through $g(E_1)$.
Denote this curve by $\overline{R}$, and denote by $\widetilde{R}$ and $R$ be the proper transforms of this curve on the surfaces~$\widetilde{S}$ and $S$, respectively.
If $\overline{R}$ is singular at  $g(E_1)$, then $R$ is a smooth rational curve such that
$R^2\leqslant\widetilde{R}^2=-3$,
which is impossible, because $S$ satisfies \eqref{equation:condition}.
Thus, we see that the curve $R$ is smooth at the point $g(E_1)$.
Then $\widetilde{R}\sim -K_{\widetilde{S}}$ and $\widetilde{R}^2=0$.
In particular, the curve $\widetilde{R}$ is a nef divisor.
On the other hand, we have $\widetilde{C}_1\cdot\widetilde{R}=1$,
because $\overline{C}_1=\overline{T}$ does not contain the point $g(E_1)$, since $m_1=0$.
Then
$$
0=K_{\widetilde{S}}^2=\widetilde{D}^{\prime\prime}\cdot\widetilde{R}=\frac{\lambda_1-a_1}{1-a_1}\widetilde{C}_1\cdot\widetilde{R}+\sum_{i=2}^n\frac{\lambda_i}{1-a_1}\widetilde{C}_i\cdot\widetilde{R}\geqslant \frac{\lambda_1-a_1}{1-a_1}\widetilde{C}_1\cdot\widetilde{R}=\frac{\lambda_1-a_1}{1-a_1},
$$
so that $a_1\geqslant\lambda_1$. This is a contradiction, since we already proved that $a_1<\lambda_1$.
\end{proof}

\begin{lemma}
\label{lemma:2}
Suppose that  $r+K_S^2=2$, and $\lambda_i>0$ for every $i$.
Then $U$ is not a cylinder.
\end{lemma}

\begin{proof}
We have $K_{\overline{S}}^2=2$, so that $\overline{S}$ is a smooth del Pezzo surface by Remark~\ref{remark:del-Pezzo}.
If $K_S^2=2$, then $r=0$ and $S\cong\overline{S}$.
In this case, the required assertion follows from Theorem~\ref{theorem:del-Pezzo}.
Thus, we may assume that $K_S^2\leqslant 1$.
Moreover, arguing as in the proof of Lemma~\ref{lemma:1}, we may assume that $C_i\ne E_j$ for all $i$~and~$j$.
Then, applying \cite[Lemma~3.1]{CheltsovParkWon1} to the log pair $(\overline{S},\overline{D})$,
we conclude that $\lambda_i\leqslant 1$ for each~$i$.

Suppose that $U=S\setminus\mathrm{Supp}(D)$ is a cylinder. Then $n\geqslant 9$ by Lemma~\ref{lemma:main-1}.
Moreover, by Lemma~\ref{lemma:main-2}, the surface $S$ contains a point $P$ such that the log pair $(S,D)$ is not log canonical at~$P$.
In the notations of Section~\ref{section:preliminaries}, the point $P$ is the point $\pi(\mathcal{C})$.
Let $\overline{P}=g(P)$. Then $(\overline{S},\overline{D})$ is not log canonical at $\overline{P}$,
because $P\not\in E_1\cup\cdots\cup E_{r}$ by Lemma~\ref{lemma:easy}.

By Theorem~\ref{theorem:del-Pezzo-tigers}, the linear system $|-K_{\overline{S}}|$ contains a curve $\overline{T}$
such that $(\overline{S},\overline{T})$ is not log canonical at the point $\overline{P}$,
and  irreducible components of the curve $\overline{T}$ are among the curves $\overline{C}_1,\ldots,\overline{C}_n$.
In particular, the curve $\overline{T}$ is singular at $\overline{P}$.
Note that this property uniquely determines the curve $\overline{T}$.
Moreover, since $\overline{S}$ is a smooth del Pezzo surface of degree $K_{\overline{S}}^2=2$,
the curve $\overline{T}$ has at most two irreducible component.
Thus, without loss of generality, we may assume that either $\overline{T}=\overline{C}_1$,
or $\overline{T}=\overline{C}_1+\overline{C}_2$ and $\lambda_1\leqslant\lambda_2$.

If $\overline{T}=\overline{C}_1$, then $\overline{T}$ has a cuspidal singularity at the point $\overline{P}$.
Likewise, if $\overline{T}=\overline{C}_1+\overline{C}_2$, then $\overline{T}$ has a tacknodal singularity at the point $\overline{P}$.
In both cases, the point $\overline{P}$ is the unique singular point of the curve $\overline{T}$.
As in the proof of Lemma~\ref{lemma:1}, for every $i\in\{1,\ldots,r\}$, let
$$
m_i=\left\{\aligned
&0\ \text{if}\ g(E_i)\not\in\overline{T},\\
&1\ \text{if}\ g(E_i)\in\overline{T}.\\
\endaligned
\right.
$$
Let $T$ be the proper transform of the curve $\overline{T}$ on the surface $S$. Then
$$
T\sim g^{*}\big(\overline{T}\big)-\sum_{i=1}^{r}m_iE_i\sim -K_{S}+\sum_{i=1}^{r}(1-m_i)E_i.
$$

If $\overline{T}=\overline{C}_1$, then $\lambda_1<1$, because
$$
2=-K_{\overline{S}}\cdot\overline{D}=\sum_{i=1}^{n}\lambda_i\Big(-K_{\overline{S}}\cdot\overline{C}_i\Big)=2\lambda_1+\sum_{i=2}^{n}\lambda_i\Big(-K_{\overline{S}}\cdot\overline{C}_i\Big)\geqslant 2\lambda_1+\sum_{i=2}^n\lambda_i>2\lambda_1.
$$
Similarly, if $\overline{T}=\overline{C}_1+\overline{C}_2$, then $\lambda_1<1$, because
$$
2=-K_{\overline{S}}\cdot\overline{D}=\lambda_1+\lambda_2+\sum_{i=3}^{n}\lambda_i\Big(-K_{\overline{S}}\cdot\overline{C}_i\Big)>\lambda_1+\lambda_2\geqslant 2\lambda_1.
$$

Let $D^\prime=\frac{1}{1-\lambda_1}D-\frac{\lambda_1}{1-\lambda_1}T$ and $\overline{D}^\prime=\frac{1}{1-\lambda_1}\overline{D}-\frac{\lambda_1}{1-\lambda_1}\overline{T}$.
If $\overline{T}=\overline{C}_1$, then
$$
D^\prime=\sum_{i=2}^n\frac{\lambda_i}{1-\lambda_1}C_i.
$$
Similarly, if $\overline{T}=\overline{C}_1+\overline{C}_2$, then
$$
D^\prime=\frac{\lambda_2-\lambda_1}{1-\lambda_1}C_2+\sum_{i=3}^n\frac{\lambda_i}{1-\lambda_1}C_i.
$$
In both cases, the divisor $D^\prime$ is effective, and its support does not contain the curve $C_1$.
On the other hand, we have
$$
D^\prime\sim_{\mathbb{Q}}-K_{S}+\sum_{i=1}^{r}\frac{a_i+(m_i-1)\lambda_1}{1-\lambda_1}E_i.
$$
Thus, if $\frac{a_i+(m_i-1)\lambda_1}{1-\lambda_1}\geqslant 0$ for every $i$,
then the log pair $(S,D^\prime)$ is not log canonical at the point $P$ by Lemma~\ref{lemma:main-2}.
Then $\overline{D}^\prime\sim_{\mathbb{Q}}-K_{\overline{S}}$ and $(\overline{S}, \overline{D}^\prime)$ is not log canonical at $\overline{P}$, which contradicts Theorem~\ref{theorem:del-Pezzo-tigers}.
Hence, at least one of the numbers $\frac{a_1+(m_1-1)\lambda_1}{1-\lambda_1},\ldots,\frac{a_r+(m_r-1)\lambda_1}{1-\lambda_1}$ must be negative.
Without loss of generality, we may assume that
$$
\frac{a_i+(m_i-1)\lambda_1}{1-\lambda_1}<0\iff i\leqslant k
$$
for some $k\leqslant r$, and $a_1\leqslant\cdots\leqslant a_k$.
Then $m_i=0$ and $a_i<\lambda_1$ for every $i\leqslant k$.

Let $D^{\prime\prime}=\frac{1}{1-a_1}D-\frac{a_1}{1-a_1}T$.
Then $D^{\prime\prime}$ is effective. Indeed, if $\overline{T}=\overline{C}_1$, then
$$
D^{\prime\prime}=\frac{\lambda_1-a_1}{1-a_1}C_1+\sum_{i=2}^n\frac{\lambda_i}{1-a_1}C_i.
$$
Similarly, if $\overline{T}=\overline{C}_1+\overline{C}_2$, then
$$
D^{\prime\prime}=\frac{\lambda_1-a_1}{1-a_1}C_1+\frac{\lambda_2-a_1}{1-a_1}C_2+\sum_{i=3}^n\frac{\lambda_i}{1-a_1}C_i.
$$
Note that $\mathrm{Supp}(D^{\prime\prime})=\mathrm{Supp}(D)$. On the other hand, we have
$$
D^{\prime\prime}\sim_{\mathbb{Q}}-K_{S}+\sum_{i=2}^{r}\frac{a_i-a_1(1-m_i)}{1-a_1}E_i.
$$
Applying Lemma~\ref{lemma:1} to $D^{\prime\prime}$, we see that $U$ is not a cylinder. This is a contradiction.
\end{proof}

\begin{lemma}
\label{lemma:3}
Suppose that  $r+K_S^2=3$, and $\lambda_i>0$ for every~$i$.
Then $U$ is not a cylinder.
\end{lemma}

\begin{proof}
Since $K_{\overline{S}}^2=3$, we see that $\overline{S}$ is a smooth cubic surface in $\mathbb{P}^3$ by Remark~\ref{remark:del-Pezzo}.
Thus, if $K_S^2=3$, then $r=0$ and $S\cong\overline{S}$ and $D\sim_{\mathbb{Q}}-K_{S}$,
so that $U=S\setminus\mathrm{Supp}(D)$ is not a cylinder by Theorem~\ref{theorem:del-Pezzo}.
Therefore, we may assume that $K_S^2\leqslant 2$.
Moreover, arguing as in the proof of Lemma~\ref{lemma:1}, we may assume that $C_i\ne E_j$ for all possible $i$~and~$j$.
Then, applying \cite[Lemma~4.1]{CheltsovParkWon1} to the log pair $(\overline{S},\overline{D})$, we conclude that $\lambda_i\leqslant 1$ for each~$i$.

Suppose that $U$ is a cylinder. Let us seek for a contradiction.
By Lemma~\ref{lemma:main-1}, we have $n\geqslant 8$ .
By Lemma~\ref{lemma:main-2}, the surface~$S$ contains a point $P$ such that $(S,D)$ is not log canonical at~$P$.
In the notations of Section~\ref{section:preliminaries}, the point $P$ is the point $\pi(\mathcal{C})$.
Let $\overline{P}=g(P)$. Then $(\overline{S},\overline{D})$ is not log canonical at $\overline{P}$,
because $P\not\in E_1\cup\cdots\cup E_{r}$ by Lemma~\ref{lemma:easy}.

Let $\overline{T}$ be the hyperplane section of the cubic surface $\overline{S}$ that is singular at $\overline{P}$.
By Theorem~\ref{theorem:del-Pezzo-tigers}, the pair $(\overline{S},\overline{T})$ is not log canonical~at~$\overline{P}$,
and all irreducible components of the curve $\overline{T}$ are among the irreducible curves $\overline{C}_1,\ldots,\overline{C}_n$.
Thus, we may assume that
either $\overline{T}=\overline{C}_1$,
or $\overline{T}=\overline{C}_1+\overline{C}_2$ and $\lambda_1\leqslant\lambda_2$,
or $\overline{T}=\overline{C}_1+\overline{C}_2+\overline{C}_3$ and $\lambda_1\leqslant\lambda_2\leqslant\lambda_3$.

If $\overline{T}=\overline{C}_1$, then $\overline{T}$ has a cuspidal singularity at the point $\overline{P}$.
Likewise, if $\overline{T}=\overline{C}_1+\overline{C}_2$, then $\overline{T}$ has a tacknodal singularity at the point $\overline{P}$.
Finally, if $\overline{T}=\overline{C}_1+\overline{C}_2+\overline{C}_3$, then the curves $\overline{C}_1$, $\overline{C}_2$ and $\overline{C}_3$ are lines that pass through the point $\overline{P}$.
Therefore, in all possible cases, the point $\overline{P}$ is the unique singular point of the curve $\overline{T}$.
As in the proofs of Lemmas~\ref{lemma:1} and \ref{lemma:1}, for every $i\in\{1,\ldots,r\}$, we let
$$
m_i=\left\{\aligned
&0\ \text{if}\ g(E_i)\not\in\overline{T},\\
&1\ \text{if}\ g(E_i)\in\overline{T}.\\
\endaligned
\right.
$$
Let $T$ be the proper transform of the curve $\overline{T}$ on the surface $S$. Then
$$
T\sim g^{*}\big(\overline{T}\big)-\sum_{i=1}^rm_iE_i\sim -K_{S}+\sum_{i=1}^r(1-m_i)E_i.
$$

We claim that $\lambda_1<1$. Indeed, if $\overline{T}=\overline{C}_1$, then
$$
3=-K_{\overline{S}}\cdot\overline{D}=\sum_{i=1}^{n}\lambda_i\Big(-K_{\overline{S}}\cdot\overline{C}_i\Big)=3\lambda_1+\sum_{i=2}^{n}\lambda_i\Big(-K_{\overline{S}}\cdot\overline{C}_i\Big)\geqslant 3\lambda_1+\sum_{i=2}^n\lambda_i>3\lambda_1,
$$
so that $\lambda_1<1$. Similarly, if $\overline{T}=\overline{C}_1+\overline{C}_2$, then $\lambda_1<1$, because
$$
3=\lambda_1\mathrm{deg}\big(\overline{C}_1\big)+\lambda_2\mathrm{deg}\big(\overline{C}_2\big)+\sum_{i=3}^{n}\lambda_i\Big(-K_{\overline{S}}\cdot\overline{C}_i\Big)>\lambda_1\Big(\mathrm{deg}\big(\overline{C}_1\big)+\mathrm{deg}\big(\overline{C}_2\big)\Big)=3\lambda_1.
$$
Finally, if $\overline{T}=\overline{C}_1+\overline{C}_2+\overline{C}_3$, then we also have $\lambda_1<1$, because
$$
3=-K_{\overline{S}}\cdot\overline{D}=\lambda_1+\lambda_2+\lambda_3+\sum_{i=4}^{n}\lambda_i\Big(-K_{\overline{S}}\cdot\overline{C}_i\Big)>\lambda_1+\lambda_2+\lambda_3\geqslant 3\lambda_1.
$$

Let $D^\prime=\frac{1}{1-\lambda_1}D-\frac{\lambda_1}{1-\lambda_1}T$ and $\overline{D}^\prime=\frac{1}{1-\lambda_1}\overline{D}-\frac{\lambda_1}{1-\lambda_1}\overline{T}$.
If $\overline{T}=\overline{C}_1$, then
$$
D^\prime=\sum_{i=2}^n\frac{\lambda_i}{1-\lambda_1}C_i.
$$
Similarly, if $\overline{T}=\overline{C}_1+\overline{C}_2$, then
$$
D^\prime=\frac{\lambda_2-\lambda_1}{1-\lambda_1}C_2+\sum_{i=3}^n\frac{\lambda_i}{1-\lambda_1}C_i.
$$
Finally, if $\overline{T}=\overline{C}_1+\overline{C}_2+\overline{C}_3$, then
$$
D^\prime=\frac{\lambda_2-\lambda_1}{1-\lambda_1}C_2+\frac{\lambda_3-\lambda_1}{1-\lambda_1}C_3+\sum_{i=4}^n\frac{\lambda_i}{1-\lambda_1}C_i.
$$
Therefore, in all cases, the divisor $D^\prime$ is effective, and its support does not contain the curve $C_1$.
On the other hand, we have
$$
D^\prime\sim_{\mathbb{Q}}-K_{S}+\sum_{i=1}^{r}\frac{a_i+(m_i-1)\lambda_1}{1-\lambda_1}E_i.
$$
Thus, if $\frac{a_i+(m_i-1)\lambda_1}{1-\lambda_1}\geqslant 0$ for every $i$,
then $(S,D^\prime)$ is not log canonical at $P$ by Lemma~\ref{lemma:main-2},
so that the log pair $(\overline{S}, \overline{D}^\prime)$ is not log canonical at $\overline{P}$,
which contradicts Theorem~\ref{theorem:del-Pezzo-tigers}, because $\overline{D}^\prime\sim_{\mathbb{Q}}-K_{\overline{S}}$
and the support of the divisor $\overline{D}^\prime$ does not contain the curve $\overline{C}_1$.
Hence, at least one number among $\frac{a_1+(m_1-1)\lambda_1}{1-\lambda_1},\ldots,\frac{a_r+(m_r-1)\lambda_1}{1-\lambda_1}$ is negative.

Without loss of generality, we may assume that
$$
\frac{a_i+(m_i-1)\lambda_1}{1-\lambda_1}<0\iff i\leqslant k
$$
for some $k\leqslant r$, and $a_1\leqslant\cdots\leqslant a_k$.
Then $m_i=0$ and $a_i<\lambda_1$ for every $i=1,\ldots,k$.

Put $D^{\prime\prime}=\frac{1}{1-a_1}D-\frac{a_1}{1-a_1}T$.
Then $D^{\prime\prime}$ is an effective divisor. Indeed, if $\overline{T}=\overline{C}_1$, then
$$
D^{\prime\prime}=\frac{\lambda_1-a_1}{1-a_1}C_1+\sum_{i=2}^n\frac{\lambda_i}{1-a_1}C_i.
$$
Similarly, if $\overline{T}=\overline{C}_1+\overline{C}_2$, then
$$
D^{\prime\prime}=\frac{\lambda_1-a_1}{1-a_1}C_1+\frac{\lambda_2-a_1}{1-a_1}C_2+\sum_{i=3}^n\frac{\lambda_i}{1-a_1}C_i.
$$
Finally, if $\overline{T}=\overline{C}_1+\overline{C}_2+\overline{C}_3$, then
$$
D^{\prime\prime}=\frac{\lambda_1-a_1}{1-a_1}C_1+\frac{\lambda_2-a_1}{1-a_1}C_2+\frac{\lambda_3-a_1}{1-a_1}C_3+\sum_{i=4}^n\frac{a_i}{1-a_1}C_i.
$$
In all cases $\mathrm{Supp}(D^{\prime\prime})=\mathrm{Supp}(D)$. On the other hand, we have
$$
D^{\prime\prime}\sim_{\mathbb{Q}}-K_{S}+\sum_{i=2}^{r}\frac{a_i-a_1(1-m_i)}{1-a_1}E_i.
$$
Applying Lemma~\ref{lemma:2} to $D^{\prime\prime}$, we see that $U$ is not a cylinder. This is a contradiction.
\end{proof}

\section{The proof}
\label{section:proof}

In this section, we prove Theorem~\ref{theorem:main} using Lemmas~\ref{lemma:1}, \ref{lemma:2} and \ref{lemma:3}.

Let $S$ be a smooth rational surface,
let $A$ be an ample $\mathbb{Q}$-divisor on $S$, and let $\mu_A$ be its Fujita invariant.
Then $K_S+\mu_AA$ is contained in the boundary of the Mori cone $\overline{\mathbb{NE}(S)}$.
Thus, the divisor $K_S+\mu_AA$ is pseudo-effective, and it is not big.
Let $\Delta_{A}$ be the smallest extremal face of the cone $\overline{\mathbb{NE}(S)}$ that contains $K_{S}+\mu_A A$,
and let $r_A$ be the dimension of this face, i.e. $r_A$ is the Fujita~rank of the divisor $A$.
To prove Theorem~\ref{theorem:main}, we have to show that $S$ does not contain $A$-polar cylinders
if $S$ satisfies \eqref{equation:condition}, and $r_A+K_S^2\leqslant 3$.

First, let us describe the Zariski decomposition of the divisor $K_{S}+\mu_A A$, 
which follows from  \cite[Theorem~1]{Sakai} or \cite{Prokhorov}.
To be precise, we have

\begin{lemma}
\label{lemma:Sakai}
There is a birational morphism $g\colon S\to\overline{S}$ such that $\overline{S}$ is smooth, and
$$
K_{S}+\mu_A A\sim_{\mathbb{Q}}g^{*}\Big(K_{\overline{S}}+\mu_A\overline{A}\Big)+\sum_{i=1}^{r}a_iE_i,
$$
where $E_1,\ldots,E_{r}$ are all $g$-exceptional curves, $a_1,\ldots,a_{r}$ are positive rational numbers, $\overline{A}=g_{*}(A)$,
the divisor $K_{\overline{S}}+\mu_A\overline{A}$ is nef, and $(K_{\overline{S}}+\mu_A\overline{A})^2=0$.
Moreover, one of the following two cases holds:
\begin{enumerate}
\item either $\overline{S}$ is a smooth del Pezzo surface, $K_{\overline{S}}+\mu_A\overline{A}\sim_{\mathbb{Q}} 0$, and $r=r_A$,

\item or there exists a conic bundle $h\colon\overline{S}\to\mathbb{P}^1$ such that
$K_{\overline{S}}+\mu_A\overline{A}\sim_{\mathbb{Q}} qF$ for a positive rational number $q$,
where $F$ is a fiber of $h$, and $r_A=\mathrm{rk}\,\mathrm{Pic}(S)-1$.
\end{enumerate}
\end{lemma}

\begin{proof}
The surface $S$ contains an irreducible curve $C$ such that
$\mu_AA\sim_{\mathbb{Q}}aC$ for some positive rational number $a$, and the singularities of the log pair $(S,aC)$ are log terminal.
Thus, we can apply Log Minimal Model Program to this log pair (see \cite{KollarMori98}).

If $K_{S}+aC\sim_{\mathbb{Q}} 0$, the required assertion is obvious.
Likewise, if $K_{S}+aC\not\sim_{\mathbb{Q}} 0$ and the divisor $K_{S}+aC$ is nef,
then $(K_{S}+aC)^2=0$, because $K_{S}+aC$ is not big by assumption.
In this case, the required assertion follows from \cite[Theorem~3.3]{KollarMori98}, because $C$ is ample.
Thus, we may assume that $K_{S}+aC$ is not nef.

If $\mathrm{rk}\,\mathrm{Pic}(S)=1$, then $S=\mathbb{P}^2$.
If $\mathrm{rk}\,\mathrm{Pic}(S)=2$, then $S$ is one of Hirzebruch surfaces.
In both cases, the required assertion is obvious.
Thus, we may assume that $\mathrm{rk}\,\mathrm{Pic}(S)\geqslant 3$.
Then, since $K_{S}+aC$ is not nef, there exists a birational map $g_1\colon S\to S_1$  that
contracts an irreducible curve $E_1$ such that $(K_{S}+aC)\cdot E_1<0$.
Since $C$ is ample, we see that $E_1\ne C$ and $K_S\cdot E_1<0$, which implies that $E_1$ is a smooth rational curve, and $E_1^2=-1$.
In particular, the surface $S_1$ is smooth.

Let $C_1=g(C)$. Then
$$
K_{S}+aC\sim_{\mathbb{Q}}g_1^{*}\big(K_{S_1}+aC_1\big)+b_1E_i
$$
for some rational number $b_1>0$.
Then $(S_1,aC_1)$ is log terminal, the divisor $aC_1$ is ample,
and the divisor $K_{S_1}+aC_1$ is contained in the boundary of the Mori cone $\overline{\mathbb{NE}(S_1)}$.
Hence, we can apply the same arguments to $K_{S_1}+aC_1$ and iterate the whole process.
Eventually, after finitely many steps, this gives us the required assertions.
\end{proof}

Now we suppose that $r_A+K_S^2\leqslant 3$.
Since $\mathrm{rk}\,\mathrm{Pic}(S)=10-K_S^2$,
the face $\Delta_{A}$ has large codimension in $\overline{\mathbb{NE}(S)}$.
Thus, by Lemma~\ref{lemma:Sakai}, the nef part of the Zariski decomposition of the divisor  $K_S+\mu_AA$ is trivial,
and there exists a birational morphism $g\colon S\to\overline{S}$ such that $\overline{S}$ is a smooth del Pezzo surface,
$g$ contracts $r_A$ smooth rational curves, and
$$
\mu_AA\sim_{\mathbb{Q}} -K_S+\sum_{i=1}^{r_A}a_iE_i,
$$
where $E_1,\ldots,E_{r_A}$ are $g$-exceptional curves,
and $a_1,\ldots,a_{r_A}$ are positive rational numbers.
Observe also that $K_{\overline{S}}^2=r_A+K_S^2$, so that $r_A+K_S^2\geqslant 1$.

Finally, we suppose that $S$ satisfies \eqref{equation:condition}.
Then~the curves $E_1,\ldots,E_{r_A}$ must be disjoint, so that $E_1^2=E_2^2=\cdots=E_{r_A}^2=-1$.
To prove Theorem~\ref{theorem:main}, we have to show that $S$ does not contain $A$-polar cylinders.
Suppose that this is not the case.
Then there is an effective $\mathbb{Q}$-divisor $D$ on the surface $S$ such that $S\setminus\mathrm{Supp}(D)$ is a cylinder,
and $D\sim_{\mathbb{Q}} A$.
This contradicts Lemmas~\ref{lemma:1}, \ref{lemma:2} and \ref{lemma:3}, because $r_A+K_S^2\in\{1,2,3\}$.

\end{document}